\documentclass[11pt,a4paper,openright,oneside]{article}
\usepackage{amsfonts, amsmath, amssymb,latexsym, amsthm, mathrsfs, enumerate}
\usepackage[english]{babel}
\usepackage{epsfig}
\usepackage{xcolor}
\usepackage{mathtools}
\usepackage[new]{old-arrows}
\usepackage{hyperref}
\usepackage{authblk}

\usepackage[margin=1.2in]{geometry}
\usepackage{graphicx}
\usepackage{listings}
\usepackage[latin1]{inputenc}

\newcommand\restrict[1]{\raisebox{-.5ex}{$|$}_{#1}}

\makeatletter
\providecommand\dotbigcup{\mathpalette\@barred\cdot}
\def\@barred#1#2{\ooalign{\hfil$#1\bigcup$\hfil\cr\hfil$#1#2$\hfil\cr}}

\makeatother

\numberwithin{equation}{section}
\newtheorem{teo}{Theorem}[section]
\newtheorem*{teo*}{Theorem}
\newtheorem*{prop*}{Proposition}
\newtheorem*{corol*}{Corollary}

\newtheorem{lemma}[teo]{Lemma}

\newtheorem{conjecture}[teo]{Conjecture}

\theoremstyle{definition}

\newtheorem{remark}[teo]{Remark}

 \theoremstyle{definition}

\def\qedblack{\hfill $\blacksquare$}

\usepackage{fancyhdr}

\title{Colouring Complete Multipartite and
Kneser-type Digraphs}
\author{Ararat Harutyunyan\thanks{\texttt{ararat.harutyunyan@lamsade.dauphine.fr}}~~~~~~~~Gil Puig i Surroca\thanks{\texttt{gil.puig-i-surroca@dauphine.eu}}}
\affil{LAMSADE, Universit\'e Paris Dauphine - PSL \\ 75775 Paris Cedex 16, France}

\begin{document}
\maketitle

\begin{abstract}
The \emph{dichromatic number} of a digraph $D$ is the smallest $k$ such that $D$ can be partitioned into $k$ acyclic subdigraphs, and the dichromatic number of an undirected graph is the maximum dichromatic number over all its orientations. Extending a well-known result of Lov\'{a}sz, we show that the dichromatic number of the Kneser graph $KG(n,k)$ is $\Theta(n-2k+2)$ and that the dichromatic number of the Borsuk graph $BG(n+1,a)$ is $n+2$ if $a$ is large enough. 
We then study the list version of the dichromatic number. We show that, for any $\varepsilon>0$ and $2\leq k\leq n^{\frac{1}{2}-\varepsilon}$, the list dichromatic number of $KG(n,k)$ is $\Theta(n\ln n)$. 
This extends a recent result of Bulankina and Kupavskii %\cite{BulankinaKupavskii2022} 
on the list chromatic number of $KG(n,k)$, where the same behaviour was observed. We also show that for any $\rho>3$, $r\geq 2$ and $m\geq\max\{\ln^{\rho}r,2\}$, the list dichromatic number of the complete $r$-partite graph with $m$ vertices in each part is $\Theta(r\ln m)$, extending a classical result of Alon.
%\cite{Alon1992}
Finally, we give a directed analogue of Sabidussi's theorem on the chromatic number
of graph products.

%We present bounds for the dichromatic number of Kneser graphs and Borsuk graphs, and for the list dichromatic number of complete multipartite graphs certain classes of Kneser graphs and complete multipartite graphs. The bounds presented are sharp up to a constant factor. Additionally, we give a directed analogue of Sabidussi's theorem on the chromatic number of graph products.
\end{abstract}

\begin{center} 
\small{\textbf{2010 Mathematics Subject Classification:} 05C15, 05C20, 05C69, 05C76}
 
\small{\textbf{Keywords:} dichromatic number,
list dichromatic number,
complete multipartite graphs, Kneser graphs, Borsuk graphs}
\end{center} 

\section{Introduction}

We consider graphs/digraphs without loops or multiple edges/arcs. They are all finite unless otherwise specified. A \emph{proper $k$-colouring} of an undirected graph $G=(V,E)$ is a mapping $f:V\rightarrow [k]=\{1,...,k\}$ such that $f^{-1}(i)$ is an independent set for every $i\in[k]$. The \emph{chromatic number} of $G$, denoted by $\chi(G)$, is the minimum $k$ for which $G$ has a proper $k$-colouring. %{\color{orange}[The analogue of independent sets in directed graphs are \emph{acyclic sets}, i.e., sets without directed cycles. This way, a]}
A \emph{proper $k$-colouring} of a digraph $D=(V,A)$ is a mapping $f:V\rightarrow [k]$ such that $f^{-1}(i)$ is acyclic (i.e.~the subdigraph induced by $f^{-1}(i)$ has no directed cycles) for every $i\in[k]$, and the \emph{dichromatic number} of $D$, denoted by $\vec\chi(D)$, is the minimum $k$ for which $D$ has a proper $k$-colouring. Note that this definition generalizes the usual colouring, in the sense that the chromatic number of a graph is equal to the dichromatic number of its corresponding bidirected digraph. The notion was introduced by Neumann-Lara in 1982 \cite{Neumann-Lara1982} and it was later rediscovered by Mohar \cite{Mohar2003}. Since then, it has been shown that many classical results hold also in this setting \cite{Betal04,HM11b,HM11a,HM11c}. However, some fundamental questions remain unanswered. The \emph{dichromatic number} of an undirected graph $G$, denoted by $\vec\chi(G)$, is the maximum dichromatic number over all its orientations. Erd{\H o}s and Neumann-Lara conjectured the following.
\begin{conjecture}\label{conj:E-NL}\emph{\cite{Erdos1979}} For every integer $k$ there exists an integer $r(k)$ such that $\vec\chi(G)\geq k$ for any undirected graph $G$ satisfying $\chi(G)\geq r(k)$.
\end{conjecture}
For instance, $r(1)=1$ and $r(2)=3$. But it is already unknown whether $r(3)$ exists. Mohar and Wu \cite{MoharWu2016} managed to prove the fractional analogue of Conjecture~\ref{conj:E-NL}. Providing further evidence for the conjecture, they showed that Kneser graphs with large chromatic number have large dichromatic number. Improving their bound, we show that the dichromatic number of Kneser graphs is of the order of their chromatic number. %{\color{blue}[We then study the dichromatic number of Borsuk graphs.$\rightarrow$maybe put it just on the last paragraph of the intro?]}

%{\color{orange}[Finding the chromatic number of Kneser graphs was an open problem for more than two decades \cite{Kneser1955,Ziegler2021}. The famous resolution by L\'ovasz \cite{Lovasz1978} was inspired by the analogy between Kneser graphs and Borsuk graphs. We show that the chromatic particularities of Borsuk graphs are preserved in the directed setting; indeed, their chromatic and dichromatic numbers coincide.]}

In the 1970s Erd{\H o}s, Rubin and Taylor~\cite{ERT1979}, and, independently, Vizing~\cite{Vizing1976}, introduced the list variant of the colouring problem, which can be carried over to the directed setting as well. A \emph{$k$-list assignment} to a graph $G$ (or a digraph $D$) with vertex set $V$ is a mapping $L:V\rightarrow\binom{\mathbb Z^+}{k}=\{X\subseteq\mathbb Z^+\mid |X|=k\}$. A colouring (a mapping) $f:V\rightarrow\mathbb Z^+$ is said to be \emph{accepted} by $L$ if $f(v)\in L(v)$ for every $v\in V$ (or just to be \emph{acceptable} if the list assignment is understood). $G$ (resp.~$D$) is \emph{$k$-list colourable} %{\color{blue}[find a good common terminology]} 
if every $k$-list assignment accepts a proper colouring. The \emph{list chromatic number} (or the \emph{choice number}) of $G$ (resp.~the \emph{list dichromatic number} of $D$), denoted by  $\chi_{\ell}(G)$ (resp.~$\vec\chi_{\ell}(D)$), is the minimum $k$ such that $G$ (resp.~$D$) is $k$-list colourable. Similarly, the \emph{list dichromatic number} of $G$, denoted by $\vec\chi_{\ell}(G)$, is the maximum list dichromatic number over all its orientations. Bensmail, Harutyunyan and Le \cite{BHL2018} gave a sample of instances where the list dichromatic number of digraphs behaves as its undirected counterpart. Recently, Bulankina and Kupavskii \cite{BulankinaKupavskii2022} determined up to a constant factor the list chromatic number of a significant fraction of Kneser graphs. We show that their list dichromatic number is of the same order as the list chromatic number. %Dense Kneser graphs have some similarities with complete multipartite graphs. The list chromatic number of complete uniform multipartite graphs was found, up to a constant factor, by Alon \cite{Alon1992}. We prove that, also in this case, the list dichromatic number is of the order of the list chromatic number.

The paper is organised as follows. In Section 2, we prove that 
Kneser graphs have large dichromatic number (at least a multiplicative constant
of their chromatic number) and we show that the general lower bound for the chromatic number of Borsuk graphs is also a general lower bound for their dichromatic number. In Section 3, we study the list dichromatic number
of the complete multipartite graph $K_{m*r}$, determining its asymptotics
by extending a classical result of Alon \cite{Alon1992}. Then we extend Bulankina and Kupavskii's result on the list chromatic number of Kneser graphs to
the list dichromatic number. Finally, in Section 4, we prove a directed analogue of Sabidussi's theorem on the chromatic number of Cartesian products of graphs.

\section{The dichromatic number of Kneser graphs and Borsuk graphs}

The \emph{Kneser graph} with parameters $n,k$, denoted by $KG(n,k)$, is the undirected graph with vertex set $\binom{[n]}{k}$ where two vertices $u,v$ are adjacent if and only if $u\cap v=\emptyset$. It is well-known \cite{Greene2002,Lovasz1978,Matousek2003} that $\chi(KG(n,k))=n-2k+2$ for $1\leq k\leq\frac{n}{2}$, as conjectured by Kneser \cite{Kneser1955,Ziegler2021}.

Mohar and Wu showed that, if $k$ is not too close to $\frac{n}{2}$, the dichromatic number of $KG(n,k)$ is unbounded as well. More precisely, they proved the following.

\begin{teo}\label{thm:MW}\emph{\cite{MoharWu2016}} For any positive integers $n,k$ with $1\leq k\leq\frac{n}{2}$ we have
\[\vec\chi(KG(n,k))\geq \left\lfloor\frac{n-2k+2}{8\log_2\frac{n}{k}}\right\rfloor.\]
\end{teo}
Note that, since $\vec\chi(G)\leq\chi(G)$ for any graph $G$, the estimate in Theorem~\ref{thm:MW} is sharp up to a constant factor when $k$ is a constant fraction of $n$. Theorem~\ref{teo:Kneser} improves this bound for slower growth rates of $k$. Note that it cannot be extended to $k=1$ due to the following result on tournaments. 
\begin{teo}\label{thm:chromatic_tournaments}\emph{\cite{EGK1991,EM1964,Harutyunyan2011}} Let $T$ be a tournament of order $n$. Then $\vec\chi(T)\leq\frac{n}{\log_2 n}(1+\text o(1))$.
\end{teo}

For the proof of Theorem~\ref{teo:Kneser}, we shall adapt Greene's proof of Kneser's conjecture~\cite{Greene2002}. The following version of Lusternik--Schnirelmann--Borsuk theorem plays a key role.

\begin{lemma}\label{lem:Greene} \emph{\cite{Greene2002}} If the sphere $\mathbb S^n$ is covered with $n+1$ sets, each of which is either open or closed, then one of the sets contains a pair of antipodal points.
\end{lemma}

From now on, by $\mathbb S^n$ we will denote the embedded unit sphere $S(0,1)=\{x\in\mathbb R^{n+1}\mid \|x\|=1\}\subseteq\mathbb R^{n+1}$ centered at the origin. The following probabilistic lemma will also be of help. %Note that $\binom{n}{k}\leq\left(\frac{en}{k}\right)^k$ by Stirling's formula.

\begin{lemma}\label{lem:dense_bipartite} Let $G$ be a graph of order $n\geq 2$ and $D$ the random orientation of $G$ obtained by orienting every edge independently with probability $\frac{1}{2}$.~Let $E_{\ell}$ be the event that there exists a subgraph of $G$ isomorphic to $K_{\ell,\ell}$ which is acyclic in $D$.~If $5\log_2 n\leq\ell$ then $\mathbb P(E_{\ell})<\frac{1}{2}$.
\end{lemma}
\begin{proof} %{\color{blue}Ararat: We don't really care about number of automorphism}
%Since $K_{\ell,\ell}$ has $2\ell!^2$ automorphisms, there are at most $\frac{1}{2}\binom{n}{2\ell}\binom{2\ell}{\ell}$ copies of $K_{\ell,\ell}$ in $G$. 
Each acyclic orientation of $K_{\ell,\ell}$ can be extended to a transitive tournament on the same vertex set, and different orientations always extend to different tournaments. Therefore, among the $2^{\ell^2}$ possible orientations of $K_{\ell,\ell}$, at most $(2\ell)!\leq(2\ell)^{2\ell}\leq n^{2\ell}$ are acyclic. Since $G$ has at most $\binom{n}{2\ell}(2\ell)!\leq n^{2\ell}$ copies of $K_{\ell,\ell}$, we have that 
$\mathbb P(E_{\ell})\leq n^{4\ell}2^{-\ell^2}\leq 2^{-\ell^2/5}\leq 2^{-5}$.
\end{proof}
%{\color{blue}[update the constant of the claim of the thm about $\vec\chi_{\ell}(K_{m*r})$?]}

\begin{teo}\label{teo:Kneser} There exist a positive integer $n_0$ such that, for all $n\geq n_0$ and $2\leq k\leq\frac{n}{2}$, we have $\vec\chi(KG(n,k))\geq\left\lfloor \frac{1}{16}\chi(KG(n,k))\right\rfloor$.
\end{teo}
\begin{proof} Let $0<c<\frac{1}{2}$ be a constant and set $t=\frac{-1}{8\log_2 c}$; we will show that $\vec\chi(KG(n,k))\geq\left\lfloor t\chi(KG(n,k))\right\rfloor$ if $c$ is smaller than a certain quantity. Picking $c=\frac{1}{4}$ will suffice, although there is some margin to choose larger values. If $cn\leq k\leq\frac{n}{2}$, then the result is implied by Theorem~\ref{thm:MW}.

Now suppose that $2\leq k\leq cn$. We assume for a contradiction that, for any given orientation of $KG(n,k)$, we can find a partition of its vertex set into $d=\lfloor t(n-2k+2)\rfloor-1$ acyclic subsets $\mathcal A_1,...,\mathcal A_d$. Let $X\subseteq\mathbb S^{d}\subseteq\mathbb R^{d+1}$ be a set of $n$ points on the unit sphere centered at the origin. We assume that these points together with the origin are in general position. In particular, there are no $d+1$ points of $X$ in a common hyperplane through the origin. The set of vertices of $KG(n,k)$ is assumed to be $\binom{X}{k}$. 

Let $s=tk+(1-t)\left(\frac{n}{2}+1\right)$ and $\ell=\left\lceil\frac{1}{d}\binom{\lfloor s\rfloor}{k}\right\rceil$ (note that $d\geq 1$; otherwise, the result is immediate). We define $U_i$ as the set of points $x\in\mathbb S^d$ for which there exist $\ell$ different vertices $A_1,...,A_{\ell}\in\mathcal A_i$ such that $x\cdot y>0$ for every $y\in A_1\cup...\cup A_{\ell}$. That is, $U_i$ is the set of poles of the open hemispheres containing all the points of $\ell$ vertices of $\mathcal A_i$.
It is clear that $U_i$ is an open set of $\mathbb S^d$. Additionally, we define $F=\mathbb S^d\setminus U_1\setminus...\setminus U_d$. By Lemma~\ref{lem:Greene}, one of the sets $U_1,...,U_d,F$ contains two antipodal points.

Suppose that $U_i$ contains two antipodal points $x,-x$. Then the hemispheres with pole $x,-x$ each contain the points of $\ell$ vertices of $\mathcal A_i$. Therefore $KG(n,k)[\mathcal A_i]$ has a subgraph isomorphic to $K_{\ell,\ell}$. By Lemma~\ref{lem:dense_bipartite}, $\ell\leq 5\log_2\binom{n}{k}\leq 5 k\log_2 n$. On the other hand, 
\[\ell\geq\frac{\binom{\lfloor s\rfloor}{k}}{d}\geq\frac{1}{n}\frac{\lfloor s\rfloor(\lfloor s\rfloor-1)...(\lfloor s\rfloor-k+1)}{k!}\geq\frac{1}{n}\left(\frac{\lfloor s\rfloor}{k}\right)^k\geq\frac{1}{n}\left(\frac{s-1}{k}\right)^k\geq\frac{1}{n}\left(\frac{(1-t)n}{2k}\right)^k.\]
We distinguish two cases.

\textbf{Case 1.} $2\leq k\leq n^{\frac{1}{5}}$.

In this case 
\[\ell\geq\frac{1}{n}\left(\frac{(1-t)n}{2k}\right)^k\geq\frac{1}{n}\left(\frac{(1-t)n^{\frac{4}{5}}}{2}\right)^k\geq n^{\frac{1}{5}}\left(\frac{(1-t)n^{\frac{1}{5}}}{2}\right)^k,\]
contradicting, when $n$ is large, that $\ell\leq 5k\log_2 n\leq 5n^{\frac{1}{5}}\log_2 n$.

\textbf{Case 2.} $n^{\frac{1}{5}}\leq k\leq cn$.

In this case
\[\ell\geq\frac{1}{n}\left(\frac{(1-t)n}{2k}\right)^k\geq\frac{1}{n}\left(\frac{1-t}{2c}\right)^{n^{\frac{1}{5}}}.\]
Provided that $1-t-2c>0$, this contradicts that $\ell\leq 5k\log_2 n\leq 5c n\log_2 n$ when $n$ is large. Note that by picking $c=\frac{1}{4}$ we have $1-t-2c=1-\frac{1}{16}-\frac{1}{2}>0$.

In conclusion, $F$ must contain two antipodal points $x,-x$. But then the hemispheres with pole $x,-x$ each contain at most $\lfloor s\rfloor-1$ points of $X$. Indeed, if one of them contained $\lfloor s\rfloor$ points, it would contain the points of $\binom{\lfloor s\rfloor}{k}>\left(\left\lceil\frac{1}{d}\binom{\lfloor s\rfloor}{k}\right\rceil-1\right)d=(\ell-1)d$ vertices, so at least $\ell$ vertices of the same colour would be involved%, meaning that its pole is not in $F$
. Hence, there are at least $n-2(\lfloor s\rfloor-1)\geq n-2s+2=t(n-2k+2)\geq d+1$ points of $X$ on the hyperplane separating the two hemispheres, contradicting the general position of $X\cup\{0\}$.
\end{proof}

Kneser's conjecture remained open for more than two decades \cite{Kneser1955,Ziegler2021}. The famous resolution by Lov\'asz~\cite{Lovasz1978} was inspired by the analogy between Kneser graphs and Borsuk graphs. Let $n$ be a natural number and $a\in(0,2)$ a real number. The \emph{Borsuk graph} with parameters $n+1$ and $a$, denoted by $BG(n+1,a)$, is the undirected graph with vertex set $\mathbb S^n=\{x\in\mathbb R^{n+1}\mid \|x\|=1\}$ where two vertices $x,y$ are adjacent if and only if $\|y-x\|\geq a$. The study of the chromatic number of Borsuk graphs can be linked with geometric packing/covering problems. If $a$ is large enough, an $(n+2)$-colouring of $BG(n+1,a)$ can be obtained by projecting the faces of an inscribed $(n+1)$-dimensional simplex. It turns out that this cannot be improved, no matter how close $a$ is to $2$. Indeed, it is known that $\chi(BG(n+1,a))\geq n+2$ for every $a\in(0,2)$, which is in fact equivalent to the Borsuk--Ulam theorem~\cite{Matousek2003}. The rest of the present section is devoted to prove that the dichromatic number of Borsuk graphs admits the same general lower bound. %{\color{gray}It is known that $\chi(BG(n+1,a))\geq n+2$ for every $a\in(0,2)$, which in fact is equivalent to the Borsuk--Ulam theorem \cite{Matousek2003}. On the other hand, if $a$ is large enough, an $(n+2)$-colouring of $BG(n+1,a)$ can be obtained by projecting the faces of an inscribed $(n+1)$-dimensional simplex. %{\color{orange}[The resemblance between Borsuk graphs and Kneser graphs was instrumental in Lov\'asz's proof of Kneser's conjecture \cite{Lovasz1978}. These similarities can be exploited to compute the dichromatic number of Borsuk graphs. In general, the exact value of the dichromatic number is difficult to find, and in the case of parameterised families of graphs we are usually happy with tight asymptotic bounds. In this case benevolent structural properties make it possible.]} {\color{blue}[too much nonsense?]}
%Regarding the dichromatic number of Borsuk graphs, the following holds.}

\begin{teo}\label{teo:Borsuk} $\vec\chi(BG(n+1,a))\geq n+2$ for any $n\geq 1$ and any $a\in(0,2)$.
\end{teo}
\begin{proof} Let us denote by $B(x,r)$ the open ball $\{y\in\mathbb R^{n+1}\mid \|y-x\|<r\}$. Let $\delta\in (0,2)$ such that every point in $B(x,\delta)\cap\mathbb S^n$ is adjacent to every point in $B(-x,\delta)\cap\mathbb S^n$ for any $x\in\mathbb S^n$. Let $\ell$ be an integer that for now remains unspecified, but that is assumed to be as large as desired. We define $m=\left\lceil\sqrt[n+1]{(\ell-1)(n+1)}\right\rceil+1\leq 2\sqrt[n+1]{(\ell-1)(n+1)}$ and $c=\frac{\delta}{m\sqrt{n+1}}$. 

An \emph{open hypercube} of $\mathbb R^{n+1}$ is the image by a rigid transformation of a product of intervals $(0,\lambda)^{n+1}\subseteq\mathbb R^{n+1}$, where $\lambda\in\mathbb R^+$. The length of its \emph{side} is $\lambda$ and the length of its \emph{longest diagonal} is its diameter (i.e.~$\lambda\sqrt{n+1}$). Let $\mathcal Q_c$ be the set of open hypercubes of side $c$ of the form $(ck_1,ck_1+c)\times...\times(ck_{n+1},ck_{n+1}+c)$ with $(k_1,...,k_{n+1})\in\mathbb Z^{n+1}$, i.e.~the ones obtained by rescaling the integer lattice by a factor of $c$. We will make use of the following easy observations about $\mathcal Q_c$.

%of side $c$ defined by the lattice $c\mathbb Z^{n+1}$. That is, the elements of $\mathcal Q$ are exactly the subsets of $\mathbb R^{n+1}$ of the form $\{(x_1,...,x_{n+1})\in\mathbb R^{n+1}\mid \forall i\in[n+1]\ ck_i<x_i<c(k_i+1)\}$ for some $(k_1,...,k_{n+1})\in\mathbb Z^{n+1}$. We will make use of the following easy observations about $\mathcal Q$.
%{\color{blue}Ararat: Definition of $Q$ is very unclear; what is an open hypercube?}

\vspace{2mm}
\noindent\textbf{Observation 1.} For every $x\in\mathbb R^{n+1}$, $B(x,\delta)$ contains at least $m^{n+1}$ hypercubes of $\mathcal Q_c$.

\vspace{2mm}
\noindent\textit{Proof.} Consider an open hypercube $Q$ of longest diagonal $\delta$ with $x\in Q$. Clearly $Q\subseteq B(x,\delta)$ and the side of $Q$ is $\frac{\delta}{\sqrt{n+1}}$. This implies the claim. \qedblack
\vspace{2mm}

\noindent\textbf{Observation 2.} $B(0,1+2\delta)$ is contained in any open hypercube $Q$ of side $2c\left\lceil\frac{1+2\delta}{c}\right\rceil$ centered at the origin. Moreover, one (in fact exactly one) such $Q$ can be obtained as the interior of the closure of the union of $\left(2\left\lceil\frac{1+2\delta}{c}\right\rceil\right)^{n+1}$ hypercubes of $\mathcal Q_c$. \qedblack
\vspace{2mm}

Let $\mathcal Q'_c\subseteq\mathcal Q_c$ be the set of $\left(2\left\lceil\frac{1+2\delta}{c}\right\rceil\right)^{n+1}$ hypercubes from Observation 2. For each $Q\in\mathcal Q'_c$ choose a point $x_Q\in Q$. Let $y_Q$ be the point where the open ray starting at the origin and passing through $x_Q$ intersects $\mathbb S^n$. Since $n\geq 1$ we can assume that the points $x_Q$ have been chosen so that $y_Q\neq y_{Q'}$ if $Q\neq Q'$. Let $Y=\{y_Q\mid Q\in\mathcal Q'_c\}$. Note that 
\[|Y|=|\mathcal Q'_c|=\left(2\left\lceil\frac{(1+2\delta)m\sqrt{n+1}}{\delta}\right\rceil\right)^{n+1}\leq\left(\frac{8(1+2\delta)\sqrt{n+1}}{\delta}\right)^{n+1}(\ell-1)(n+1).\]

\vspace{2mm}
\noindent\textbf{Observation 3.} For every $x\in\mathbb S^{n}$, $B(x,\delta)$ contains at least $m^{n+1}$ points of $Y$.

\vspace{2mm}
\noindent\textit{Proof.} Since $B((1+\delta)x,\delta)\subseteq B(0,1+2\delta)$, all hypercubes of $\mathcal Q_c$ intersecting $B((1+\delta)x,\delta)$ are in $\mathcal Q'_c$. Hence, by Observation~1, $B((1+\delta)x,\delta)$ contains $m^{n+1}$ hypercubes of $\mathcal Q'_c$. The points in $Y$ corresponding to these hypercubes all lie in $B(x,\delta)$. %{\color{blue}[include the geometrical justification of the last sentence?]}
\qedblack
\vspace{2mm}

We now consider the finite induced subgraph $H=BG(n+1,a)[Y]$ of $BG(n+1,a)$. It will be enough to show that $\vec\chi(H)\geq n+2$. Let us assume for a contradiction that each orientation of $H$ admits a partition of $Y$ into $n+1$ acyclic subsets $Y_1,...,Y_{n+1}$. For $i\in[n+1]$ let $U_i=\{x\in\mathbb S^n\mid |B(x,\delta)\cap Y_i|\geq\ell\}$. Clearly, $U_i$ is an open set of $\mathbb S^n$. Moreover, $\mathbb S^n=U_1\cup...\cup U_{n+1}$. Indeed, otherwise $B(x,\delta)$ would contain at most $(\ell-1)(n+1)<m^{n+1}$ points of $Y$ for some $x\in\mathbb S^n$, contradicting Observation~3. Therefore, by Lemma~\ref{lem:Greene}, $U_i$ contains two antipodal points $x$ and $-x$ for some $i\in[n+1]$.

By the choice of $\delta$, we know that in $H$ there is a copy of $K_{\ell,\ell}$ of colour $i$. Now, $5\log_2 |Y|\leq\ell$ if $\ell$ is large enough. By Lemma~\ref{lem:dense_bipartite}, there is an orientation of $H$ such that every copy of $K_{\ell,\ell}$ in $H$ has a directed cycle, a contradiction.
\end{proof}

\section{The list dichromatic number of Kneser graphs and complete multipartite graphs}\label{sec:lists}

The goal of this section is to study the list dichromatic number of complete
multipartite graphs and Kneser graphs. In the first subsection, we study the 
list dichromatic number of complete multipartite graphs, obtaining
tight upper and lower bounds (up to a multiplactive factor). Following this,
in the second subsection, we consider Kneser graphs.

\subsection{List dichromatic number of complete multipartite graphs}

We denote by $K_{m*r}$ the complete $r$-partite graph with $m$ vertices on each part. Answering a question of Erd{\H o}s, Rubin and Taylor \cite{ERT1979}, Alon determined, up to a constant factor, the list chromatic number of $K_{m*r}$.

\begin{teo}\label{thm:complete_multipartite}\emph{\cite{Alon1992}} There exist two positive constants $c_1$ and $c_2$ such that for every $m\geq 2$ and for every $r\geq 2$
\[c_1 r\ln m\leq\chi_{\ell}(K_{m*r})\leq c_2 r\ln m.\]
\end{teo}

More precise results were obtained in \cite{GK2006}. Adapting Alon's proof, we find an analogous bound for the list dichromatic number of $K_{m*r}$ when $r\geq 2$ and $m\geq\max\{\ln^{\rho}r,2\}$, for any $\rho>3$ (Theorem~\ref{teo:directed_multipartite}). We remark that it is known that $\vec{\chi}(K_{m*r}) = r$, when $m$
is sufficiently large (see~\cite{HHH23}; see also~\cite{Erdos1979}).
%{\color{blue}[Some comment on the limitations when $m$ is very small? Result on tournaments? (although is not exactly that)]} 
We will use the following probabilistic result, which is a consequence of the Hoeffding--Azuma
inequality.
\begin{teo}\label{teo:simple_concentration} \emph{(Simple Concentration Bound, \cite{MolloyReed2002})} Let $X$ be a random variable determined by $n$ independent trials, and satisfying the property that changing the outcome of any single trial can affect $X$ by at most $c$. Then
\[\mathbb P(|X-\mathbb EX|>t)\leq 2e^{-\frac{t^2}{2c^2 n}}.\]
\end{teo}

We will also invoke the following elementary fact.
\begin{remark}\label{lem:calculus} Let $a\in\mathbb R^+$. The function $f:(a,\infty)\rightarrow\mathbb R$ defined by $f(x)=\left(1-\frac{a}{x}\right)^x$ is increasing.
\end{remark}
\begin{proof} $f'(x)=f(x)\left(\ln\left(1-\frac{a}{x}\right)+\frac{a}{x-a}\right)\geq f(x)\left(\ln\frac{x-a}{x}+\ln\frac{x}{x-a}\right)=0$.
\end{proof}

\begin{teo}\label{teo:directed_multipartite} For every $\rho>3$ there exist constants $c_1,c_2\in\mathbb R^+$ such that if $r\geq 2$ and $m\geq\max\{\ln^{\rho} r,2\}$ then \[c_1 r\ln m\leq\vec\chi_{\ell}(K_{m*r})\leq c_2 r\ln m.\]
\end{teo}
\begin{proof} Let $V_1,...,V_r$ be the parts of $K_{m*r}$. The upper bound is implied by Theorem~\ref{thm:complete_multipartite}. For the lower bound, we can assume that $m$ is large enough; otherwise, we get the job done by picking a suitable $c_1$.

\vspace{2mm}
\textbf{Claim.} There is a constant $c$ and an orientation $D$ of $K_{m*r}$ such that, if $\ell\geq c\ln(rm)$, each subgraph of $K_{m*r}$ isomorphic to $K_{\ell}$ or to $K_{\ell,\ell}$ has a directed cycle in $D$.

\vspace{2mm}
\noindent\textit{Proof.} We orient the edges of $K_{m*r}$ at random, independently and with probability $\frac{1}{2}$. Let $E$ (resp.~$E'$) be the event that each subgraph of $K_{m*r}$ isomorphic to $K_{\ell}$ (resp.~$K_{\ell,\ell}$) has a directed cycle. By Lemma~\ref{lem:dense_bipartite}, $\mathbb P(E),\mathbb P(E')>\frac{1}{2}$ if $c$ is sufficently large. Hence $\mathbb P(E\cap E')>0$. %{\color{blue}[look if $E'$ can be precised in order to get better conditions on $\rho$]}
\qedblack
\vspace{2mm}

Let $k=\lfloor Cr\ln m\rfloor$, where $0<C\leq 1$ is a constant for now unspecified. We start by showing that there exists an assignment of $k$-lists from a palette $\mathscr C$ of $\lfloor r\ln m\rfloor$ colours such that, for any given set $A\subseteq\mathscr C$ of at most $\frac{4}{3}\ln m$ colours, each part has at least $\frac{1}{2}m^{1-\delta}$ vertices that avoid the colours from $A$ on their lists, where $\delta=2C\ln 5$.

We assign to each vertex $v$ of $D$ a random $k$-list $L(v)$ chosen independently and uniformly among the $\binom{|\mathscr C|}{k}$ possible $k$-lists. Given $i\in[r]$ and $A\subseteq\mathscr C$, consider the random variable $X_{i,A}=|\{v\in V_i\mid L(v)\cap A=\emptyset\}|$. Note that there are exactly $\binom{|\mathscr C|-|A|}{k}$ $k$-lists avoiding the colours in $A$. Devoting ourselves to the case $|A|=\left\lfloor\frac{4}{3}\ln m\right\rfloor$, we have that
\[\mathbb E X_{i,A}=m\frac{\binom{|\mathscr C|-|A|}{k}}{\binom{|\mathscr C|}{k}}\geq m\left(\frac{|\mathscr C|-|A|-k}{|\mathscr C|-k}\right)^k=m\left(1-\frac{|A|}{|\mathscr C|-k}\right)^k\]
\[\geq m\left(1-\frac{\frac{4}{3}\ln m}{(1-C)r\ln m -1}\right)^{Cr\ln m}\geq m\left(1-\frac{4}{5}\right)^{2C\ln m}=m^{1-\delta}\]
if $m$ is large enough and $C$ is not too large, using Remark~\ref{lem:calculus}. By the Simple Concentration Bound (Theorem~\ref{teo:simple_concentration}), \[\mathbb P(X_{i,A}<\frac{1}{2}m^{1-\delta})\leq\mathbb P(|X_{i,A}-\mathbb E X_{i,A}|>\frac{1}{2}m^{1-\delta})\leq 2e^{-\frac{1}{8}m^{1-2\delta}}.\] 
Let $E$ be the event that $X_{i,A}<\frac{1}{2}m^{1-\delta}$ for some $i\in[r]$ and $A\subseteq\mathscr C$ with $|A|\leq\frac{4}{3}\ln m$. We have that
\[\mathbb P(E)\leq r\binom{|\mathscr C|}{\left\lfloor\frac{4}{3}\ln m\right\rfloor}2e^{-\frac{1}{8}m^{1-2\delta}}\leq (r\ln m)^{\frac{4}{3}\ln m+1}2e^{-\frac{1}{8}m^{1-2\delta}}\] 
\[\leq 2e^{\left(m^{\frac{1}{\rho}}+\ln\ln m\right)\left(\frac{4}{3}\ln m+1\right)-\frac{1}{8}m^{1-2\delta}}\leq 2e^{2m^{\frac{1}{\rho}}\ln m-\frac{1}{8}m^{1-2\delta}}\]
%\[\mathbb P(E)\leq r\binom{|\mathscr C|}{\left\lfloor\frac{4}{3}\ln m\right\rfloor}2e^{-\frac{1}{8}m^{1-2\delta}}\leq r\left(e\frac{\lfloor r\ln m\rfloor}{\left\lfloor\frac{4}{3}\ln m\right\rfloor}\right)^{\left\lfloor\frac{4}{3}\ln m\right\rfloor}2e^{-\frac{1}{8}m^{1-2\delta}}\] 
%\[\leq r(er)^{\frac{4}{3}\ln m}2e^{-\frac{1}{8}m^{1-2\delta}}\leq 2e^{5\ln r\ln m-\frac{1}{8}m^{1-2\delta}}\leq 2e^{5m^{\frac{1}{\rho}}\ln m-\frac{1}{8}m^{1-2\delta}}\]
if $m$ is large enough. Consequently, if $\delta<\frac{1}{2}(1-\frac{1}{\rho})$ and $m$ is large enough, there exists a list assignment $L'$ satisfying the desired property. This is the assignment that we are going to use.

Now let $f$ be a proper colouring of $D$. We claim that there exists a set of indices $I\subseteq [r]$ of size at least $\frac{3r}{4}$ such that $|f(V_i)|\leq 4c\ln^2(rm)$ for each $i\in I$. Indeed, if more than $\frac{r}{4}$ parts are coloured with more than $4c\ln^2(rm)$ colours each, then one of the colours appears on more than $\frac{cr\ln^2(rm)}{|\mathscr C|}\geq c\frac{\ln^{2}(rm)}{\ln m}\geq c\ln(rm)$ parts. By the choice of $D$, $f$ is not proper, a contradiction.

For each $i\in[r]$ define the set $A_i=\{\gamma\in\mathscr C\mid |V_i\cap f^{-1}(\gamma)|\geq c\ln(rm)\}$. We claim that if $f$ is acceptable then $|A_i|>\frac{4}{3}\ln m$ for every $i\in I$. Indeed, otherwise, by the choice of the lists, at least $\frac{1}{2}m^{1-\delta}$ vertices of $V_i$ have been coloured with colours not from $A_i$. Thus one of these colours is used at least 
\[\frac{\frac{1}{2}m^{1-\delta}}{4c\ln^2(rm)}\leq c\ln(rm)\] 
times on $V_i$. If $m$ is large enough, this implies that \[m^{1-\delta}\leq 8c^2\ln^3(rm)\leq 8c^2(m^{\frac{1}{\rho}}+\ln m)^3\leq 9c^2m^{\frac{3}{\rho}}.\] If we further assume that $\delta<1-\frac{3}{\rho}$, we get a contradiction when $m$ is large. Therefore $|A_i|>\frac{4}{3}\ln m$ for every $i\in I$.

Now, by the choice of $D$, the sets $A_1,...,A_r$ are mutually disjoint. But then \[|\mathscr C|\geq\sum_{i=1}^r |A_i|\geq\sum_{i\in I} |A_i|>\frac{4}{3}|I|\ln m\geq r\ln m\geq |\mathscr C|.\] This contradiction shows that there is no acceptable proper colouring for the $k$-list assignment $L'$.
\end{proof}

%It would be desirable to know what happens for the rest of values of $m,r$. What is clear is that the bound of Theorem~\ref{teo:directed_multipartite} is not valid in general. Indeed, if $m\leq\ln r$ then Theorem~\ref{thm:list_chromatic_tournaments} implies that $\vec\chi_{\ell}(K_{m*r})\leq \vec\chi_{\ell}(K_{mr})\leq Cr$ for some constant $C$. {\color{blue}[This could be nonsense]}
We do not know what happens for other values of $m,r$. What is clear is that Theorem~\ref{teo:directed_multipartite} is not valid in general. Indeed, if $m\leq\ln r$ then the following theorem %{\color{blue}(Theorem~\ref{thm:list_chromatic_tournaments}) $\leftarrow$[is this needed?]} 
implies that $\vec\chi_{\ell}(K_{m*r})\leq \vec\chi_{\ell}(K_{mr})\leq cr$ for some constant $c$.
\begin{teo}\label{thm:list_chromatic_tournaments}\emph{\cite{BHL2018}} Let $T$ be a tournament of order $n$. Then $\vec\chi_{\ell}(T)\leq\frac{n}{\log_2 n}(1+\text o(1))$.
\end{teo}

\subsection{List dichromatic number of Kneser graphs}

Here we investigate the list dichromatic number of Kneser graphs. 
The list chromatic number of Kneser graphs was recently studied by Bulankina and Kupavskii. They proved the following two results.

\begin{teo}\label{thm:BK1}\emph{\cite{BulankinaKupavskii2022}} For any positive integers $n,k$ with $1\leq k\leq\frac{n}{2}$ we have $\chi_{\ell}(KG(n,k))\leq n\ln\frac{n}{k}+n$.
\end{teo}

\begin{teo}\label{thm:BK2}\emph{\cite{BulankinaKupavskii2022}} For every 
$\varepsilon\ > 0$, there exists a constant $c_{\varepsilon} > 0$ such that $\chi_{\ell}(KG(n,k))\geq c_{\varepsilon}n\ln n$ for all $n,k$ with $2\leq k\leq n^{\frac{1}{2}-\varepsilon}$.
\end{teo}

%{\color{blue}[Comment that in fact they have better (precise) constants?]} 
However, good lower bounds for larger values of $k$ are still unknown. Clearly,
the upper bound of Theorem \ref{thm:BK1} trivially generalises to the dichromatic number.
The rest of the subsection is devoted to the proof of the directed analogue of Theorem~\ref{thm:BK2};
that is, we show that the lower bound can be strengthened to digraphs. The proof is achieved by a sequence of lemmas, which involve the argument of Bulankina and Kupavskii, as well as ideas from Mohar and Wu \cite{MoharWu2016}.
As in Theorem \ref{thm:BK2}, we do not know if the bound on $k$ can be extended to $2\leq k\leq n^{1-\varepsilon}$ for an arbitrarily small $\varepsilon>0$.
%is necessary. 

\begin{teo}\label{thm:list_dichromatic_Kneser} For every $\varepsilon > 0$ there exists a constant $c_{\varepsilon} > 0$ such that $\vec\chi_{\ell}(KG(n,k))\geq c_{\varepsilon}n\ln n$ for all $n,k$ with $2\leq k\leq n^{\frac{1}{2}-\varepsilon}$.
\end{teo}

Let $G=(V,E)$ be a graph, $\mathcal C$ a collection of subsets of $V$ and $s,t$ positive integers. We say that $\mathcal C$ is an \emph{$(s,t)$-collection} of $V$ if
\begin{enumerate}[(i)]
   \item $|\mathcal C|\leq s$;
   \item $\forall C\in\mathcal C\ \ \, |C|\leq t$.
\end{enumerate}
Given a list assignment $L$, we denote by $U = \cup_{v \in V(G)} L(v)$ the total set of colours, referred to as the \emph{palette}, and we set $u=|U|$. The partitions $P$ of $V$ considered in the sequel will always have $u$
(not necessarily non-empty) parts (we always implicitly or explicitly assume that the partitions are \emph{acyclic}, i.e., that each part of $P$ induces an acyclic digraph). It will be convenient to regard as distinct any two partitions arising from different colourings. Thus, partitions will be thought as indexed by $U$ (but for simplicity we will continue calling them just ``partitions"). 
%A: I actually don't think we need the above "distinction comment"

We say that a partition $P$ of $V$ is \emph{covered} by an $(s,t)$-collection $\mathcal C$ 
(or that $\mathcal C$ is an \emph{$(s,t)$-cover} of $P$) if each part determined by $P$ is contained in some $C\in\mathcal C$.  Let $P=(P_1,...,P_u)$ be a partition. 
We say that a list assignment $L$ of $G$ \emph{accepts} $P$ if, for every $i\in[u]$ and every $v \in P_i$, $i \in L(v)$. Otherwise, we say that $L$ \emph{rejects} $P$.

%{\color{gray}In what follows $u\geq\ell_1\geq\ell_2$ are integers, as are $n,s,t$. The total set of colours, referred to as the \emph{palette}, will have cardinality $u$. Define the function $g(\ell_1,\ell_2,n,s,t,u):=s^u e^{-\frac{n}{2}2^{-\frac{4\ell_2 tu}{(\ell_1-\ell_2)n}}}$.}

In what follows, $\ell_1$ and $\ell_2$ are integers. Define
the function $g(\ell_1,\ell_2,n,s,t,u):=s^u e^{-\frac{n}{2}2^{-\frac{4\ell_2 tu}{(\ell_1-\ell_2)n}}}$.

\begin{lemma}\label{lem:BK_general} Let $G=(V,E)$ be an undirected graph of order $n$, $\mathcal C$ an $(s,t)$-collection of $V$ and $\mathcal P$ the family of partitions of $V$ covered by $\mathcal C$. Let $L_1$ be an $\ell_1$-list assignment for $G$ from a palette of $u$ colours and $L_2$ a random $\ell_2$-list assignment for $G$ where, for every $v\in V$, $L_2(v)$ is chosen independently and equiprobably among $\binom{L_1(v)}{\ell_2}$. If $4tu\leq (\ell_1-\ell_2)n$, then \[\mathbb P(L_2\text{ accepts some }P\in\mathcal P)<
g(\ell_1,\ell_2,n,s,t,u)=s^u e^{-\frac{n}{2}2^{-\frac{4\ell_2 tu}{(\ell_1-\ell_2)n}}}.\]
\end{lemma}
\begin{proof} Let $\mathrm C=(C_1,...,C_u)\in\mathcal C^u$ be any $u$-tuple of elements of $\mathcal C$. For every $v\in V$, let $r_{\mathrm C}(v)$ be the number of indices $i\in[u]$ such that $v\in C_i$. Consider the subset of vertices $W_{\mathrm C}=\{v\in V\mid r_{\mathrm C}(v)\leq\frac{2tu}{n}\}$. We claim that $|W_{\mathrm C}|>\frac{1}{2}n$. Indeed, otherwise \[tu\geq\sum_{i=1}^u|C_i|=\sum_{v\in V}r_{\mathrm C}(v)\geq\sum_{v\in V\setminus W_{\mathrm C}}r_{\mathrm C}(v)>tu.\] 
Moreover, for any $v\in W_{\mathrm C}$ the probability $p_{\mathrm C}(v)$ that $v\notin\bigcup_{i\in L_2(v)} C_i$ is at least
\[\frac{\binom{\ell_1-r_{\mathrm C}(v)}{\ell_2}}{\binom{\ell_1}{\ell_2}}=\prod_{k=1}^{\ell_2}\frac{\ell_1-\ell_2-r_{\mathrm C}(v)+k}{\ell_1-\ell_2+k}\geq\left(1-\frac{r_{\mathrm C}(v)}{\ell_1-\ell_2}\right)^{\ell_2}\]
\[\geq\left(1-\frac{2tu}{(\ell_1-\ell_2)n}\right)^{\ell_2}\geq\left(\frac{1}{2}\right)^{\frac{4\ell_2 tu}{(\ell_1-\ell_2)n}}=:p,\]
using Remark~\ref{lem:calculus} and the inequality $4tu\leq (\ell_1-\ell_2)n$. Therefore, the probability that there is some $u$-tuple $\mathrm C=(C_1,...,C_u)$ of elements of $\mathcal C$ such that $v\in\bigcup_{i\in L_2(v)}C_i$ for every $v\in V$ is at most
\[\sum_{\mathrm C\in{\mathcal C}^u}\,\prod_{v\in W_\mathrm C}(1-p_\mathrm C(v))<s^u\left(1-p\right)^{\frac{1}{2}n}\leq s^u e^{-\frac{1}{2}np}.\]
%{\color{blue}A: is $C^u$ defined? We can just replace $\Sigma$ by $\binom{s}{u}$ if need be. More importantly, strictly speaking we probably do not need to say this, but should we not say that any $P$ has exactly $u$ (not necessarily non-empty) parts.} 
%The result follows from the fact that every $P\in\mathcal P$ is covered by $\mathcal C$.
Since every $P\in\mathcal P$ is covered by $\mathcal C$, each of the %{\color{gray}(at most)}
$u$ parts of any such $P$ is contained in some $C\in\mathcal C$, so the result follows.
\end{proof}

Let $G,H$ be graphs. The \emph{tensor product} $G\times H$ of $G$ and $H$ is the graph with vertex set $V(G)\times V(H)$ where two vertices $(v,x)$ and $(w,y)$ are adjacent if and only if $v,w$ are adjacent in $G$ and $x,y$ are adjacent in $H$. %The \emph{Cartesian product} $G\square H$ is the graph with vertex set $V(G)\times V(H)$ where two vertices $(v,x)$ and $(w,y)$ are adjacent precisely when $\{v,w\}\in E(G)$ and $x=y$, or when $\{x,y\}\in E(H)$ and $v=w$. 
The tensor product of complete graphs $K_n\times K_n$ is going to play an auxiliary role; we denote it by $G_n$. Given $S\subseteq V(G_n)$, we call $\pi_1(S)$ and $\pi_2(S)$ the projection of $S$ to the first and second coordinate, respectively. The \emph{rows} (resp.~\emph{columns}) of $S$ are the subsets of $S$ of the form $S\cap(\{i\}\times [n])$ (resp.~$S\cap([n]\times\{i\})$), where $i\in[n]$. Now we give some properties of $G_n$.  %Now we are going to prove some properties about the complements of square rook graphs. Set $G_n=(K_n\square K_n)^{\mathrm c}$ {\color{blue}[note that $G_n=K_n\times K_n$, think what's more convenient]}. Given $S\subseteq V(G_n)$, we call $\pi_1(S)$ and $\pi_2(S)$ the projection of $S$ to the first and second coordinates, respectively.

\begin{lemma}\label{lem:bipartite_in_rook} %There is a constant $c_1\in(12,\infty)$ such that the following holds. 
For any $n\geq 2$, there is an orientation $D_n$ of $G_n$ (resp.~of $K_2\times G_n$) such that, for every $S,T\subseteq V(G_n)$ satisfying 
\begin{enumerate}[i)]
   %\item $|S|,|T|\geq c_1\ln n$ and
   \item $|S|,|T|\geq 30\ln n$ and
   \item $\pi_i(S)\cap\pi_i(T)=\emptyset$ for $i\in\{1,2\}$,
\end{enumerate}
the subdigraph of $D_n$ induced by $S\cup T$ (resp.~by $(\{1\}\times S)\cup(\{2\}\times T)$) has a directed cycle.
\end{lemma}
\begin{proof} Since in $G_n[S\cup T]$ (resp.~in ($K_2\times G_n)[(\{1\}\times S)\cup(\{2\}\times T)]$) all edges between $S$ and $T$ (resp.~between $\{1\}\times S$ and $\{2\}\times T$) are present, the conclusion follows from Lemma~\ref{lem:dense_bipartite}. %Indeed, taking $c_1=\frac{3C_1}{\ln 2}$ suffices, where $C_1>4$ is the constant from the mentioned lemma.
\end{proof}
%{\color{blue}[$c_1=30$, $c'_1=2^13$, $5c_1$ is substituted by $124$]}

We define an $(s_n,t_n)$-collection $\mathcal C_{G_n}$ of $V(G_n)$ as follows. Let $\mathcal L_{G_n}=\{\{i\}\times[n]\mid i\in[n]\}\cup\{[n]\times \{i\}\mid i\in[n]\}$ be the set of 
rows and columns of $V(G_n)$ and 
\[\mathcal Q_{G_n}=\begin{cases}\{A\times B\mid A,B\in\binom{[n]}{\lfloor 124\ln n\rfloor}\} &\text{if } 1\leq\lfloor 124\ln n\rfloor\leq n\\ \{V(G_n)\} &\text{otherwise}.\end{cases}\]
We set $\mathcal C_{G_n}=\{L\cup Q\mid L\in\mathcal L_{G_n},\ Q\in\mathcal Q_{G_n}\}$. Note that $|\mathcal C_{G_n}|\leq s_n:=\max\{1,2n\binom{n}{\lfloor 124\ln n\rfloor}^2\}$ and $|C|\leq t_n:=n+\lfloor 124\ln n\rfloor^2$ for any $C\in\mathcal C_{G_n}$.

\begin{lemma}\label{lem:rook_cover} There is an orientation $D_n$ of $G_n$ such that $\mathcal C_{G_n}$ covers all acyclic partitions of $D_n$.
\end{lemma}
\begin{proof} It can be assumed that $1\leq\lfloor 124\ln n\rfloor\leq n$. 
Let $D_n$ be the orientation from Lemma~\ref{lem:bipartite_in_rook}, and let $S$ be an acyclic set of $D_n$. Assume for a contradiction that $S$ is not contained in any $C\in\mathcal C_{G_n}$. Let $L$ be 
%the set of vertices of $S$ forming its largest row,
the largest row of $S$, or its largest column if it is larger than its largest row, and let $S'=S\setminus L$. Then $S'$ is not contained in any $Q\in\mathcal Q_{G_n}$, so $|\pi_i(S')|>124\ln n>90\ln n+2$ for some $i\in\{1,2\}$.
%It will be enough to show that the set $S'$ resulting from removing $L$ from $S$ is included in some $Q\in\mathcal Q_{G_n}$. Assume for a contradiction that this is not the case. Then $|\pi_i(S')|>124\ln n>90\ln n+2$ for some $i\in\{1,2\}$. 
Assume that $i=1$ (if $i=2$, the argument below is repeated with rows instead of columns). Let $L'$ be the largest column of $S'$. We distinguish three cases. We will show that,
in each case, we can find two sets in $S$ satisfying the hypotheses of Lemma \ref{lem:bipartite_in_rook}. This
will yield a contradiction since $S$ is acyclic.
%{Let $L'$ be 
%the set of vertices forming the largest row of $S'$,
%the largest row of $S'$, or its largest column if that one is larger than its largest row. We distinguish two cases.} {\color{blue}A: When we say row we mean some $i \times n$ or a
%subset of $i \times n$. I think we mean subset, and we should define row/column by saying this. G: Ok}

%\textbf{Case 1.} $|L'|>60\ln n+2$.

%Recall that $|L|\geq |L'|$. Therefore we can find a subset of $L$ and a subset of $L'$ satisfying the hypotheses of Lemma~\ref{lem:bipartite_in_rook}. But then $S$ is not acylic, a contradiction.

%\textbf{Case 2.} $|L'|\leq 60\ln n+2$.

%{\color{blue}A: This is vague. Can't you just say that in this case you are inside some $A \times B$
%without resorting to a contradiction argument by assuming $|\pi_i(S')|>124\ln n>90\ln n+2$ for some $i\in\{1,2\}$}

%In this case we can find two subsets of $S'$ satisfying the hypotheses of Lemma~\ref{lem:bipartite_in_rook}, and we reach the same contradiction.

%{\color{blue}Assume that $i=1$ (if $i=2$, repeat the argument below with rows instead of columns). Let $L''$ be the largest column of $S'$. [I have to revise this proof]}

\textbf{Case 1.} $|L'|>60\ln n+2$.

Recall that $|L|\geq |L'|$. Therefore we can find a subset of $L$ and a subset of $L'$ satisfying the hypotheses of Lemma~\ref{lem:bipartite_in_rook}.

\textbf{Case 2.} $60\ln n+2\geq |L'|\geq 30\ln n$.

Since $|\pi_i(S')|>90\ln n+2$, we can find a subset $T$ of $S'$ such that $T$ and $L'$ satisfy the hypotheses of Lemma~\ref{lem:bipartite_in_rook}.

\textbf{Case 3.} $|L'|\leq 30\ln n$.

Let $\{T_1,...,T_k\}$ be a minimal set of columns of $S'$ satisfying  $|\pi_i(\bigcup^k_{j=1} T_j)|\geq 30\ln n$. By minimality, $|\pi_i(\bigcup^k_{j=1} T_j)|\leq 60\ln n$. Hence, as in Case 2, we can find a subset $T$ of $S'$ such that $T$ and $ \bigcup^k_{j=1} T_j$ satisfy the hypotheses of of Lemma~\ref{lem:bipartite_in_rook}.

In any case, Lemma~\ref{lem:bipartite_in_rook} yields a cycle in $S$, the desired contradiction.
\end{proof}

We are now in position to determine the order of $\vec\chi_{\ell}(KG(n,k))$ when $k$ is bounded by a constant.

\begin{lemma}\label{lem:k_bounded} There is a constant $c\in\mathbb R^+$ such that $\vec{\chi}_\ell(KG(n,k))\geq c\frac{n}{k}\ln\frac{n}{k}$ for every $2\leq k\leq\frac{n}{2}$.
\end{lemma}
\begin{proof} First note that $G_{\lfloor\frac{n}{k}\rfloor}$ is isomorphic to a subgraph of $KG(n,k)$. Indeed, if we take $\left\lfloor\frac{n}{k}\right\rfloor+1$ pairwise disjoint subsets $I,J_1,...,J_{\lfloor\frac{n}{k}\rfloor}\subseteq [n]$ with $|I|=\left\lfloor\frac{n}{k}\right\rfloor$ and $|J_1|=...=|J_{\lfloor\frac{n}{k}\rfloor}|=k-1$, then the set of vertices $S=\{\{i\}\cup J_j\mid i\in I,\ 1\leq j\leq\left\lfloor\frac{n}{k}\right\rfloor\}\subseteq\binom{[n]}{k}$ induces a copy of $G_{\lfloor\frac{n}{k}\rfloor}$. 

Thus it suffices to show that $\vec\chi_{\ell}(G_{\tilde n})\geq c\tilde n\ln\tilde n$ for some $c > 0$. Assume that $\tilde n$ is large enough. Given $G_{\tilde n}$, consider the orientation $D_{\tilde n}$ from Lemma~\ref{lem:rook_cover}; we know that $\mathcal C_{G_{\tilde n}}$ covers all acyclic partitions of $D_{\tilde n}$. Let  $u_{\tilde n}=\ell_{1,\tilde n}=\lfloor\tilde n\ln\tilde n\rfloor$ and $\ell_{2,\tilde n}=\lfloor cu_{\tilde n}\rfloor$, where $c < 1$ is a positive constant to be defined later. Let $L_{1,\tilde n}$ be the canonical $\ell_{1,\tilde n}$-list assignment to $D_{\tilde n}$ (i.e.~$L_{1,\tilde n}(v)=[u_{\tilde n}]$ for every $v\in V(D_{\tilde n})$). 
%By Lemma~\ref{lem:rook_cover}, it will be enough to apply Lemma~\ref{lem:BK_general} with $\mathcal C_{G_{\tilde n}}$ and the trivial $\ell_{1,\tilde n}$-list assignement $L_{1,\tilde n}$ (i.e., $L_{1,\tilde n}(v)=[u_{\tilde n}]$ for every $v\in V(G_{\tilde n})$) {\color{blue}[This probably has to be explained better]}. 
It is clear that $4t_{\tilde n} u_{\tilde n}\leq(\ell_{1,\tilde n}-\ell_{2,\tilde n})\tilde n^2$ and
\[\ln g(\ell_{1,\tilde n},\ell_{2,\tilde n},{\tilde n}^2,s_{\tilde n},t_{\tilde n},u_{\tilde n})\leq 330\tilde n\ln^3 \tilde n-\frac{1}{2}{\tilde n}^{2-\frac{8c}{(1-c)\log_2 e}}<0\]
if $\tilde n$ is large enough and $c$ has been chosen so that  $\frac{8c}{(1-c)\log_2 e}<1$. Now, by Lemma~\ref{lem:BK_general}, $\vec\chi_{\ell}(D_{\tilde n})>\ell_{2,\tilde n}$.
\end{proof}

%As a result of Mohar and Wu suggests \cite{MoharWu2016}, the tensor product of a graph $G$ with a dense graph $H$ can easily achieve a dichromatic number close to $\chi(G)$. We will show that this is also the case when lists are involved. 

The previous lemma handles the case when $k$ is small. For larger values of $k$, we need
to modify the definition of a cover. Let $H=(V,E)$ be a graph, $\mathcal C$ an $(s,t)$-collection of $V$ and $\lambda\in\mathbb R^+$. Consider the graph $K_2\times H$ and one of its orientations $D$. We say that an acyclic partition $P$ of $V(D)$ is \emph{semicovered} by the pair $(\mathcal C,\lambda)$ (or that $(\mathcal C,\lambda)$ is an \emph{$(s,t)$-semicover} of $P$) if for every acyclic set $S=(\{1\}\times S_1)\cup(\{2\}\times S_2)\in P$ either $S_1\subseteq C_1$ and $S_2\subseteq C_2$ for some $C_1,C_2\in\mathcal C$, or $S_i\subseteq C$ and $|S_i|<\lambda$ for some $i\in\{1,2\}$ and some $C\in\mathcal C$.

\begin{lemma}\label{lem:MoharWu_general} Let $G,H$ be graphs. Let $m_G$ be the size of $G$ and $n_H$ the order of $H$. Let $D$ be an orientation of $K_2\times H$ and $(\mathcal C,\lambda)$ an $(s,t)$-semicover of all acyclic partitions of $D$. Let  $\ell_1,\ell_2$ be positive integers such that $8t\ell_1\leq (\ell_1-\ell_2)n_H$, $m_G\, g^2(\ell_1,\ell_2,n_H,s,t,2\ell_1)<1$ and $\lambda\ell_1\leq n_H$. If $\chi_{\ell}(G)>\ell_1$, then $\vec\chi_{\ell}(G\times H)>\ell_2$.
\end{lemma}
\begin{proof} Suppose that $\vec\chi_{\ell}(G\times H)\leq \ell_2$. Let $L_1$ be any $\ell_1$-list assignment for $G$. Consider a random $\ell_2$-list assignment $L_2$ for $G\times H$, where for each $v\in V(G)$ and each $x\in\{v\}\times V(H)$ $L_2(x)$ is chosen independently and equiprobably among $\binom{L_1(v)}{\ell_2}$. For each edge $\{v,w\}$ of $G$, let $\mathcal C_{\{v,w\}}=\{(\{v\}\times C_1)\cup(\{w\}\times C_2)\mid C_1,C_2\in\mathcal C\}$. We orient the subgraph induced by $\{v,w\}\times V(H)$ according to $D$ (in any of the two possible ways). This results in an orientation of $G\times H$ that we will call $\overrightarrow{G\times H}$. By applying Lemma~\ref{lem:BK_general} to $(G\times H)[\{v,w\}\times V(H)]$ with a palette of size $u=2\ell_1$, we see that the probability that ${L_2}\restrict{\{v,w\}\times V(H)}$ accepts some partition covered by $\mathcal C_{\{v,w\}}$ is smaller than $g(\ell_1,\ell_2,2n_H,s^2,2t,2\ell_1)=g^2(\ell_1,\ell_2,n_H,s,t,2\ell_1)$. 
Therefore, the probability that this happens for some edge $\{v,w\}$ of $G$ is less than $m_G\, g^2(\ell_1,\ell_2,n_H,s,t,2\ell_1)< 1$. Thus we can find a $\ell_2$-list assignment $L'_2$ for $G\times H$ such that, for every $\{v,w\}\in E(G)$, ${L'_2}\restrict{\{v,w\}\times V(H)}$ rejects all partitions of $\{v,w\}\times V(H)$ covered by $\mathcal C_{\{v,w\}}$.

Since $\vec\chi_{\ell}(G\times H)\leq\ell_2$, $\overrightarrow{G\times H}$ has a colouring $f'_2$ which is accepted by $L'_2$ and produces no monochromatic cycles. Let us define a colouring $f_1$ for $G$ as
\[f_1(v)=\begin{cases}\gamma_v\!\!\! & \text{if }\exists\gamma\ \, \forall C\in\mathcal C\ \ {(f'_2)}^{-1}(\gamma)\cap(\{v\}\times V(H))\nsubseteq\{v\}\times C \text{, where $\gamma_v$ is any such $\gamma$} \\ \gamma^+_v\!\!\! & \text{otherwise, where }\gamma^+_v \text{ is any }\gamma\text{ maximizing } |{(f'_2)}^{-1}(\gamma)\cap(\{v\}\times V(H))|.\end{cases}\] 
Note that $f_1(v) \in L_1(v)$, for each $v \in V(G)$. We will show that $f_1(v)$ is a proper colouring of $G$, and this contradiction will finish the proof.

Let $\{v,w\}$ be an edge of $G$, and suppose for a contradiction that $f_1(v)=f_1(w)$. Since 
%{\color{blue}[}$f'_2$ is accepted by $L'_2${\color{blue}]$\rightarrow$[}
${L'_2}\restrict{\{v,w\}\times V(H)}$ rejects all partitions of $\{v,w\}\times V(H)$ covered by $\mathcal C_{\{v,w\}}$,
%{\color{blue}]}
either $f_1(v)=\gamma_v$ or $f_1(w)=\gamma_w$.
%{\color{blue}A: Don't you rather mean that $L_2'$ rejects all partition etc. G: Yes, it seems better}
Without loss of generality, we can assume that $f_1(w)=\gamma_w$. Since $f'_2$ produces no monochromatic cycles, if $f_1(v)=\gamma_v$ then $(\mathcal C,\lambda)$ does not semicover all acyclic partitions of $D$, a contradiction. On the other hand, if $f_1(v)=\gamma_v^+$ then $|{(f'_2)}^{-1}(\gamma_v^+)\cap(\{v\}\times V(H))|\geq\frac{n_H}{\ell_1}\geq\lambda$, also contradicting that $(\mathcal C,\lambda)$ semicovers all acyclic partitions of $D$.

Therefore, we have found a proper colouring $f_1$ accepted by the $\ell_1$-list assignment $L_1$. Since $L_1$ was arbitrary we conclude that $\chi_{\ell}(G)\leq\ell_1$, ending the proof.
\end{proof}

\begin{lemma}\label{lem:rook_cover2} For every $n$ there is an orientation $D_n$ of $K_2\times G_n$ such that all acyclic partitions of $D_n$ are semicovered by $(\mathcal C_{G_n},2^{13}\ln^2 n)$.
\end{lemma}
\begin{proof} The proof is similar to that of Lemma~\ref{lem:rook_cover}. Let $D_n$ be the orientation of $K_2\times G_n$ from Lemma~\ref{lem:bipartite_in_rook}. Assume that $1\leq \lfloor 124\ln n\rfloor\leq n$; otherwise the lemma is trivial. Let $S=(\{1\}\times S_1)\cup(\{2\}\times S_2)$ be an acyclic set in $D_n$. Assume that $S_1$ is not contained in any $C\in\mathcal C_{G_n}$. As in the proof of Lemma~\ref{lem:rook_cover} we can find in $S_1$ two sets $L_1,L_1'\subseteq S_1$ satisfying the hypotheses of Lemma~\ref{lem:bipartite_in_rook}. We can assume that $|L_1|=|L'_1|=\lceil 30\ln n\rceil$. 

We argue by contradiction. Suppose that $|S_2|\geq 2^{13}\ln^2 n$ or that $S_2$ is not contained in any $C\in\mathcal C_{G_n}$. Let $T=\{(j_1,j_2)\in S_2\mid j_1\notin\pi_1(L_1\cup L'_1)\text{ or }j_2\notin\pi_2(L_1\cup L'_1)\}$. We claim that $|T|\geq 60\ln n$. Let us consider two cases.

\textbf{Case 1.} $|S_2|\geq 2^{13}\ln^2 n$.

We have that $|\pi_1(L_1\cup L'_1)|,|\pi_2(L_1\cup L'_1)|\leq 2\lceil 30\ln n\rceil\leq 64\ln n$. Therefore $|T|\geq |S_2|-64^2\ln^2 n\geq 60\ln n$.

\textbf{Case 2.} $S_2\nsubseteq C$ for any $C\in\mathcal C_{G_n}$.

In this case, $S_2\nsubseteq Q$ for any $Q\in\mathcal Q_{G_n}$. Therefore, $|\pi_i(S_2)|>124\ln n$ for some $i\in\{1,2\}$. Since $|\pi_i(L_1\cup L'_1)|\leq 2\lceil 30\ln n\rceil$, we have that $|T|\geq 124\ln n-2\lceil 30\ln n\rceil\geq 60\ln n$.

%Hence $|T|\geq 60\ln n$ in any case. Let $T_1=\{(j_1,j_2)\in T\mid j_1\notin\pi_1(L_1),\,j_2\notin\pi_2(L_1)\}$ {\color{blue} and $T'_1=T\setminus T_1$}. If $|T_1|\geq 30\ln n$, we apply Lemma~\ref{lem:bipartite_in_rook} with $L_1$ and $T_1$ and we find that $(\{1\}\times L_1)\cup(\{2\}\times T_1)$ induces a directed cycle, a contradiction. Otherwise, $|T\setminus T_1|\geq 30\ln n$. Observe that $\pi_i(T\setminus T_1)\cap\pi_i(L'_1)=\emptyset$ for $i\in\{1,2\}$
%{\color{blue}A: expand a little}. Hence, by Lemma~\ref{lem:bipartite_in_rook}, $(\{1\}\times L'_1)\cup(\{2\}\times(T\setminus T_1))$ induces a directed cycle.

Hence $|T|\geq 60\ln n$ in any case. Let $T_1=\{(j_1,j_2)\in T\mid j_1\notin\pi_1(L_1),\,j_2\notin\pi_2(L_1)\}$ and $T'_1=T\setminus T_1$. Note that $T'_1=\{(j_1,j_2)\in T\mid j_1\notin\pi_1(L'_1),\,j_2\notin\pi_2(L'_1)\}$ by the definition of $T$. Applying Lemma~\ref{lem:bipartite_in_rook} yields the desired contradiction. Indeed, if $|T_1|\geq 30\ln n$, then $(\{1\}\times L_1)\cup(\{2\}\times T_1)$ has a directed cycle, and if otherwise $|T'_1|\geq 30\ln n$, then $(\{1\}\times L'_1)\cup(\{2\}\times T'_1)$ has a directed cycle.
%{\color{gray} If $|T_1|\geq 30\ln n$, we apply Lemma~\ref{lem:bipartite_in_rook} with $L_1$ and $T_1$ and we find that $(\{1\}\times L_1)\cup(\{2\}\times T_1)$ induces a directed cycle, a contradiction. Otherwise, $|T'_1|\geq 30\ln n$. Observe that $\pi_i(T\setminus T_1)\cap\pi_i(L'_1)=\emptyset$ for $i\in\{1,2\}$}
%{\color{blue}A: expand a little. G: Ok. Now it's better, but it may still not be enough.} {\color{gray}Hence, by Lemma~\ref{lem:bipartite_in_rook}, $(\{1\}\times L'_1)\cup(\{2\}\times(T\setminus T_1))$ induces a directed cycle.}
\end{proof}

\begin{proof}[Proof of Theorem \ref{thm:list_dichromatic_Kneser}] We fix $\varepsilon$ and assume that $n$ is large enough (to deal with the case of $n$ small, we just make sure that $c_{\varepsilon}$ is small enough). If $k$ is bounded by a constant, then Lemma~\ref{lem:k_bounded} does the job. Therefore we can assume that $k\geq 4$. Note that $KG(\lfloor\frac{n}{2}\rfloor,k-2)\times G_{\lfloor\frac{n}{4}\rfloor}$ is a subgraph of $KG(n,k)$. Indeed, consider, for any positive integers $n_1,n_2,k_1,k_2$ satisfying $2k_1\leq n_1$, $2k_2\leq n_2$, $n_1+n_2\leq n$ and $k_1+k_2=k$, the set of vertices in $KG(n,k)$ of the form \[S=\left\{\{i_1,...,i_{k_1},j_1,...,j_{k_2}\}\ \left|\  \{i_1,...,i_{k_1}\}\in\binom{[n_1]}{k_1},\ \{j_1,...,j_{k_2}\}\in\binom{n_1+[n_2]}{k_2}\right.\right\};\] the subgraph induced by $S$ is isomorphic to $KG(n_1,k_1)\times KG(n_2,k_2)$. By the proof of Lemma~\ref{lem:k_bounded}, if $k_2\geq 2$ then $G_{\lfloor\frac{n_2}{k_2}\rfloor}$ is isomorphic to a subgraph of $KG(n_2,k_2)$, so we can just take $n_1=\lfloor\frac{n}{2}\rfloor$, $n_2=\lceil\frac{n}{2}\rceil$, $k_1=k-2$ and $k_2=2$.

Note that $KG(\lfloor\frac{n}{2}\rfloor,k-2)$ has $m_n=\frac{1}{2}\binom{\lfloor\frac{n}{2}\rfloor}{k-2}\binom{\lfloor\frac{n}{2}\rfloor-k+2}{k-2}$ edges. Provided that $n$ is sufficiently large, we can find an $\varepsilon'\in(0,1)$ such that $n^{\frac{1}{2}-\varepsilon}\leq {\lfloor\frac{n}{2}\rfloor}^{\frac{1}{2}-\varepsilon'}$. Let $c_{\varepsilon'}$ be the corresponding constant from Theorem~\ref{thm:BK2}; let $c'_{\varepsilon}\in(0,1)$ be a constant for now unspecified, $\ell_{1,n}=\lfloor c_{\varepsilon'}\lfloor\frac{n}{2}\rfloor\ln\lfloor\frac{n}{2}\rfloor\rfloor-1$ and $\ell_{2,n}=\lfloor c'_{\varepsilon}\ell_{1,n}\rfloor$. We apply Lemma~\ref{lem:MoharWu_general} with the $(s_{\lfloor\frac{n}{4}\rfloor},t_{\lfloor\frac{n}{4}\rfloor})$-semicover $(\mathcal C_{G_{\lfloor\frac{n}{4}\rfloor}},2^{13}\ln^2\lfloor\frac{n}{4}\rfloor)$ of the family of acyclic partitions of $D_{\lfloor\frac{n}{4}\rfloor}$, the orientation of $K_2\times G_{\lfloor\frac{n}{4}\rfloor}$ from Lemma~\ref{lem:rook_cover2}. Clearly 
\[8t_{\lfloor\frac{n}{4}\rfloor}\ell_{1,n}\leq (\ell_{1,n}-\ell_{2,n})\left\lfloor\frac{n}{4}\right\rfloor^2,\] 
\[2^{13}\ln^2\left\lfloor\frac{n}{4}\right\rfloor\, \ell_{1,n}\leq \left\lfloor\frac{n}{4}\right\rfloor^2\text{ and}\]
\[\ln m_n+2\ln g(\ell_{1,n},\ell_{2,n},\left\lfloor\frac{n}{4}\right\rfloor^2,s_{\lfloor\frac{n}{4}\rfloor},t_{\lfloor\frac{n}{4}\rfloor},2\ell_{1,n})\leq 2k\ln n+660c_{\varepsilon'} n\ln^3 n-\frac{1}{25}n^{2-\frac{50c_{\varepsilon'}c'_{\varepsilon}}{(1-c'_{\varepsilon})\log_2 e}}< 0\]
if $n$ is large enough and $c'_{\varepsilon}$ has been chosen so that $\frac{50c_{\varepsilon'}c'_{\varepsilon}}{(1-c'_{\varepsilon})\log_2 e}<1$. Thus Lemma~\ref{lem:MoharWu_general} and Theorem~\ref{thm:BK2} imply that $\vec\chi_{\ell}(KG(\lfloor\frac{n}{2}\rfloor,k-2)\times G_{\lfloor\frac{n}{4}\rfloor})>\ell_{2,n}$ if $n$ is large enough.
\end{proof}

\section{Sabidussi's theorem}\label{sec:sabidussi}

Tensor products are one of the leitmotifs of Section~\ref{sec:lists}. Let us now take a look at another type of graph product. Let $G,H$ be graphs (resp.~digraphs). The \emph{Cartesian product} of $G$ and $H$ is the graph  (resp.~digraph) $G\mathbin\square H$ with vertex set $V(G)\times V(H)$ where there is an edge between $(u,x)$ and $(v,y)$ (resp.~an arc from $(u,x)$ to $(v,y))$ if and only if either $u=v$ and $\{x,y\}\in E(H)$ (resp.~and $(x,y)\in A(H)$), or $x=y$ and $\{u,v\}\in E(G)$ (resp.~and $(u,v)\in A(G)$). A well-known theorem of Sabidussi \cite{Sabidussi1957} states that for any two graphs $G$ and $H$ the chromatic number of its Cartesian product is $\chi(G\mathbin\square H)=\max\{\chi(G),\chi(H)\}$. His proof can be adapted to show an analogous result for digraphs. 
%Let $G,H$ be digraphs. We define its Cartesian product $G\square H$ as the digraph with vertex set $V(G)\times V(H)$ and where there is an arc from $(g_1,h_1)$ to $(g_2,h_2)$ if and only if either $g_1=g_2$ and $(h_1,h_2)$ is an arc of $H$, or $h_1=h_2$ and $(g_1,g_2)$ is an arc of $G$. 

\begin{teo}\label{thm:sabidussi} Let $G$ and $H$ be digraphs. Then $\vec\chi(G\mathbin\square H)=\max\{\vec\chi(G),\vec\chi(H)\}$.
\end{teo}
\begin{proof} Let $N=\max\{\vec\chi(G),\vec\chi(H)\}$. Note that both $G$ and $H$ are isomorphic to a subdigraph of $G\mathbin\square H$. Therefore, $\vec\chi(G\mathbin\square H)\geq N$.

Now, let $f_G,f_H$ be $N$-colourings of $G,H$. Let $f:V(G\mathbin\square H)\rightarrow [N]$ be the $N$-colouring of $G\mathbin\square H$ defined by $f(g,h)\equiv f_G(g)+f_H(h)\mod N$.
We claim that $f$ is a proper colouring of $G\mathbin\square H$. We argue by contradiction. Let $(g_1,h_1),...,$ $(g_1,h_{s_1}),(g_2,h_{s_1+1}),...,(g_2,h_{s_2}),...,$ $(g_r,h_{s_{r-1}+1}),...,(g_r,h_{s_r})$ be the successive vertices of a monochromatic cycle, where $g_i\neq g_{i+1}$ and $h_{s_i}=h_{s_i+1}$ for $i\in[r-1]$, and $g_{r}\neq g_1$ and $h_{s_r}=h_1$. The fact that $f$ is constant on these vertices implies that $f_H(h_1)=...=f_H(h_{s_1})=f_H(h_{s_1+1})=...=f_H(h_{s_r})$. If not all of $h_1,...,h_{s_r}$ are identical, we have found a monochromatic cycle in $H$. If they are all identical, then $r\geq 2$, and we can similarly see that $f_G(g_1)=...=f_G(g_r)$, yielding a monochromatic cycle in $G$. Therefore, $\vec\chi(G\mathbin\square H)\leq N$.
\end{proof}

Hedetniemi's conjecture proposes a similar statement for tensor products: given any two graphs $G$ and $H$, the chromatic number of its tensor product is (conjectured to be) $\chi(G\times H)=\min\{\chi(G),\chi(H)\}$. That was refuted by Shitov \cite{Shitov2019}; at the time, the conjecture had been standing for more than five decades. Before the first counterexamples were known, its directed version was formulated. Given two digraphs $G$ and $H$, its tensor product $G\times H$ is defined to be the digraph with vertex set $V(G)\times V(H)$ where there is an arc from $(g_1,h_1)$ to $(g_2,h_2)$ if and only if $(g_1,g_2)$ is an arc of $G$ and $(h_1,h_2)$ is an arc of $H$.
%Another relevant topic to discuss is Hedetniemi's conjecture. Recall that it states that, for any two graphs $G$ and $H$, the chromatic number of its tensor product is $\chi(G\times H)=\min\{\chi(G),\chi(H)\}$. After resisting for more than fifty years, it was finally refuted by Shitov \cite{Shitov2019}. Before the first counterexamples were known, its directed version was formulated. Given two digraphs $G$ and $H$, its tensor product $G\times H$ is defined to be the digraph with vertex set $V(G)\times V(H)$ and where there is an arc from $(g_1,h_1)$ to $(g_2,h_2)$ if and only if $(g_1,g_2)$ is an arc of $G$ and $(h_1,h_2)$ is an arc of $H$.
\begin{conjecture} \label{conj:directed_Hedetniemi}\emph{\cite{Harutyunyan2011}} Let $G$ and $H$ be digraphs. Then $\vec\chi(G\times H)=\min\{\vec\chi(G),\vec\chi(H)\}$.
\end{conjecture}
Note that counterexamples to Conjecture~\ref{conj:directed_Hedetniemi} can be obtained by taking any counterexample to Hedetniemi's conjecture and replacing all its edges with bidirected arcs. But what happens if $G,H$ are oriented graphs? In \cite{Harutyunyan2011} it was proved that Conjecture~\ref{conj:directed_Hedetniemi} holds when $\min\{\vec\chi(G),\vec\chi(H)\}\leq 2$.

On the positive side, Zhu proved the fractional version of Hedetniemi's conjecture~\cite{Zhu11}. We wonder if this result can be generalized to the dichromatic number of digraphs, if the fractional dichromatic number of a digraph is defined in the natural way. 

\textit{Update note. After the completion of the first version of this paper, it came to our attention that Theorem~\ref{thm:sabidussi} has already been proved by Costa and Silva~\cite{CostaSilva2024}. The question of whether Conjecture~\ref{conj:directed_Hedetniemi} holds for oriented graphs has been recently answered in the negative by Picasarri-Arrieta~\cite{Picasarri-Arrieta2024}.}

\textbf{Acknowledgements.} We thank St\'{e}phane Bessy and St\'{e}phan 
Thomass\'{e} for many previous conversations around Conjecture \ref{conj:E-NL}, which
motivated us to start this paper. We also wish to thank Denis Cornaz for exciting and encouraging discussions
throughout the writing of this manuscript. The authors are supported by ANR grant
21-CE48-0012 DAGDigDec (DAGs and Digraph Decompositions). GPS was also partially supported by the Spanish Ministerio de Ciencia e Innovaci\'on through
grant PID2019-194844GB-I00.

\emph{This paper is an extended version of~\cite{HP2023}.}

\bibliographystyle{abbrv}
\bibliography{biblio}

\end{document}